\documentclass[11pt,headings=small,BCOR=15mm]{scrartcl}

\usepackage[sumlimits]{amsmath}
\usepackage{amssymb,amsfonts,amsthm}
\usepackage[english]{babel}
\usepackage[latin1]{inputenc}
\usepackage[T1]{fontenc}
\usepackage{lmodern}
\usepackage{graphicx}
\usepackage{scrpage2}
\usepackage{hyperref}
\usepackage{setspace}
\usepackage{textcomp}
\usepackage{enumerate}
\usepackage{xfrac}
\usepackage{helvet}
\usepackage{bbm}

\newtheorem{thm}{Theorem}[section]
\newtheorem{prp}[thm]{Proposition}
\newtheorem{cor}[thm]{Corollary}
\newtheorem{lem}[thm]{Lemma}
\theoremstyle{definition}

\def\Cx{\mathbb{C}}
\def\R{\mathbb{R}}
\def\N{\mathbb{N}}
\def\Q{\mathbb{Q}}

\def\D{\mathbb{D}}
\def\Dd{\mathbb{D}}
\def\B{\mathbb{B}}
\def\F{\mathbb{F}}
\def\Fb{\mathbb{F}}
\def\1{\mathbbm{1}}

\def\Dn{\mathfrak{D}}
\def\P{\mathcal{P}}
\def\C{\mathcal{C}}
\def\V{\mathcal{V}}
\def\CS{\mathcal{CS}}

\newcommand\dint{{\rm d}}
\newcommand{\im}{{\rm i}}

\DeclareMathOperator{\wal}{wal}

\DeclareMathOperator{\dg}{deg}
\DeclareMathOperator{\err}{err}
\DeclareMathOperator{\e}{e}

\allowdisplaybreaks[1]
\clearscrheadfoot
\ihead{\headmark}
\ohead{\thepage}

\begin{document}
\pagestyle{scrheadings}
\onehalfspacing

\newlength{\fixboxwidth}
\setlength{\fixboxwidth}{\marginparwidth}
\addtolength{\fixboxwidth}{-7pt}
\newcommand{\fix}[1]{\marginpar{\fbox{\parbox{\fixboxwidth}{\raggedright\tiny #1}}}}

\title{Quasi-Monte Carlo methods for integration of functions with dominating mixed smoothness in arbitrary dimension}
\author{Lev Markhasin\thanks{Research of the author was supported by a scholarship of Ernst Ludwig Ehrlich Studienwerk.}\\
        \tiny{Friedrich-Schiller-Universit\"at Jena, e-mail: lev.markhasin@uni-jena.de}}
\date{\today}
\maketitle

\begin{abstract}
In a celebrated construction, Chen and Skriganov gave explicit examples of point sets achieving the best possible $L_2$-norm of the discrepancy function. We consider the discrepancy function of the Chen-Skriganov point sets in Besov spaces with dominating mixed smoothness and show that they also achieve the best possible rate in this setting. The proof uses a $b$-adic generalization of the Haar system and corresponding characterizations of the Besov space norm. Results for further function spaces and integration errors are concluded.
\end{abstract}

\noindent{\footnotesize {\it 2010 Mathematics Subject Classification.} Primary 11K06,11K38,42C10,46E35,65C05.\\
{\it Key words and phrases.} discrepancy, Chen-Skriganov point set, dominating mixed smoothness, quasi-Monte Carlo, Haar system, numerical integration.}\\[5mm]
\textit{Acknowledgement:} The author wants to thank Aicke Hinrichs, Hans Triebel and Dmitriy Bilyk.

\section{Introduction}
Let $N$ be some positive integer and $\P$ a point set in the unit cube $[0,1)^d$ with $N$ points. Then the discrepancy function $D_{\P}$ is defined as
\begin{align}
D_{\P}(x) = \frac{1}{N} \sum_{z \in \P} \chi_{[0,x)}(z) - \left| [0,x) \right|.
\end{align}
By $|[0,x)| = x_1 \cdot \ldots \cdot x_d$ we denote the volume of the rectangular box $[0,x) = [0,x_1) \times \ldots \times [0,x_d)$ where $x = (x_1,\ldots,x_d) \in [0,1)^d$ while $\chi_{[0,x)}$ is the characteristic function of $[0,x)$. So the discrepancy function measures the deviation of the number of points of $\P$ in $[0,x)$ from the fair number of points $N |[0,x)|$ which would be achieved by a (practically impossible) perfectly uniform distribution of the points of $\P$, normalized by the total number of points.

Usually one is interested in calculating the norm of the discrepancy function in some Banach space of functions on $[0,1)^d$ to which the discrepancy function belongs. A very well known result refers to the space $L_2([0,1)^d)$. It was proved by Roth in \cite{R54}. There exists a constant $c_1 > 0$ such that, for any $N \geq 1$, the discrepancy function of any point set $\P$ in $[0,1)^d$ with $N$ points satisfies
\[ \left\| D_{\P} | L_2 \right\| \geq c_1 \, \frac{(\log N)^\frac{d-1}{2}}{N}. \]

The currently best known values for the constant $c_1$ can be found in \cite{HM11}. Furthermore, there exists a constant $c_2 > 0$ such that, for any $N \geq 1$, there exists a point set $\P$ in $[0,1)^d$ with $N$ points that satisfies
\[ \left\| D_{\P} | L_2 \right\| \leq c_2 \, \frac{(\log N)^\frac{d-1}{2}}{N}. \]
This result is known for dimension $2$ from \cite{D56} (Davenport), for dimension $3$ from \cite{R79} (Roth) and for arbitrary dimension from \cite{R80} (Roth). Only Davenport's result has been proved by an explicit construction while for higher dimensions probabilistic methods were used until Chen and Skriganov found explicit constructions for arbitrary dimension in \cite{CS02}. Results for the constant $c_2$ can be found in \cite{FPPS10}.

Both bounds were extended to $L_p$-spaces for any $1 < p < \infty$. In the case of the lower bound the reference is \cite{S77} (Schmidt) while for the upper bound it is \cite{C80} (Chen).

As general references for studies of the discrepancy function we refer to the recent monographs \cite{DP10} and \cite{NW10} as well as \cite{M99}, \cite{KN74} and \cite{B11}.

Until recently other norms than $L_p$-norms weren't studied a lot in the context of discrepancy. Triebel started the study of the discrepancy function in other function spaces like Sobolev, Besov and Triebel-Lizorkin spaces with dominating mixed smoothness in \cite{T10b} and \cite{T10a}. In \cite{Hi10} Hinrichs proved sharp upper bounds for the norms in Besov spaces with dominating mixed smoothness in dimension $2$. In \cite{M12a} the author of this work proved these upper bounds for a much larger class of point sets and also for other function spaces with dominating mixed smoothness. Triebel's result was that for all $1 \leq p,q \leq \infty$ and $r \in \R$ satisfying $\frac{1}{p} - 1 < r < \frac{1}{p}$ and $q < \infty$ if $p = 1$ and $q > 1$ if $p = \infty$ there exist constants $c_1, c_2 > 0$ such that, for any $N \geq 2$, the discrepancy function of any point set $\P$ in $[0,1)^d$ with $N$ points satisfies
\[ \left\| D_{\P} | S_{pq}^r B([0,1)^d) \right\| \geq c_1 \, N^{r-1} (\log N)^{\frac{d-1}{q}}, \]
and, for any $N \geq 2$, there exists a point set  $\P$ in $[0,1)^d$ with $N$ points such that
\[ \left\| D_{\P} | S_{pq}^r B([0,1)^d) \right\| \leq c_2 \, N^{r-1} (\log N)^{(d-1)(\frac{1}{q} + 1 - r)}. \]
Hinrichs' result closed this gap for $d = 2$, we will mention it later.

We mention some definitions from \cite{T10a} which are most important for our purpose.

Let $\mathcal{S}(\R^d)$ denote the Schwartz space and $\mathcal{S}'(\R^d)$ the space of tempered distributions on $\R^d$. For $f \in \mathcal{S}'(\R^d)$, we denote by $\mathcal{F}f$ the Fourier transform of $f$. Let $\varphi_0 \in \mathcal{S}(\R^d)$ satisfy $\varphi_0(t) = 1$ for $|t| \leq 1$ and $\varphi_0(t) = 0$ for $|t| > \frac{3}{2}$. Let
\[ \varphi_k(t) = \varphi_0(2^{-k} t) - \varphi_0(2^{-k + 1} t) \]
where $t \in \R, \, k \in \N$ and
\[ \varphi_k(t) = \varphi_{k_1}(t_1) \ldots \varphi_{k_d}(t_d) \]
where $k = (k_1,\ldots,k_d) \in \N_0^d, \, t = (t_1,\ldots,t_d) \in \R^d$.
The functions $\varphi_k$ are a dyadic resolution of unity since
\[ \sum_{k \in \N_0^d} \varphi_k(x) = 1 \]
for all $x \in \R^d$. The functions $\mathcal{F}^{-1}(\varphi_k \mathcal{F} f)$ are entire analytic functions for any $f \in \mathcal{S}'(\R^d)$.

Let $0 < p,q \leq \infty$ and $r \in \R$. The Besov space with dominating mixed smoothness $S_{pq}^r B(\R^d)$ consists of all $f \in \mathcal{S}'(\R^d)$ with finite quasi-norm
\[ \left\| f | S_{pq}^r B(\R^d) \right\| = \left( \sum_{k \in \N_0^d} 2^{r (k_1 + \ldots + k_d) q} \left\| \mathcal{F}^{-1}(\varphi_k \mathcal{F} f) | L_p(\R^d) \right\|^q \right)^{\frac{1}{q}} \]
with the usual modification if $q = \infty$.

Let $0 < p < \infty$, $0 < q \leq \infty$ and $r \in \R$. The Triebel-Lizorkin space with dominating mixed smoothness $S_{pq}^r F(\R^d)$ consists of all $f \in \mathcal{S}'(\R^d)$ with finite quasi-norm
\[ \left\| f | S_{pq}^r F(\R^d) \right\| = \left\| \left( \sum_{k \in \N_0^d} 2^{r (k_1 + \ldots + k_d) q} \left| \mathcal{F}^{-1}(\varphi_k \mathcal{F} f)(\cdot) \right|^q \right)^{\frac{1}{q}} | L_p(\R^d) \right\| \]
with the usual modification if $q = \infty$.

Let $\mathcal{D}([0,1)^d)$ consist of all complex-valued infinitely differentiable functions on $\R^d$ with compact support in the interior of $[0,1)^d$ and let $\mathcal{D}'([0,1)^d)$ be its dual space of all distributions in $[0,1)^d$. The Besov space with dominating mixed smoothness $S_{pq}^r B([0,1)^d)$ consists of all $f \in \mathcal{D}'([0,1)^d)$ with finite quasi-norm
\[ \left\| f | S_{pq}^r B([0,1)^d) \right\| = \inf \left\{ \left\| g | S_{pq}^r B(\R^d) \right\| : \: g \in S_{pq}^r B(\R^d), \: g|_{[0,1)^d} = f \right\}. \]
The Triebel-Lizorkin space with dominating mixed smoothness $S_{pq}^r F([0,1)^d)$ consists of all $f \in \mathcal{D}'([0,1)^d)$ with finite quasi-norm
\[ \left\| f | S_{pq}^r F([0,1)^d) \right\| = \inf \left\{ \left\| g | S_{pq}^r F(\R^d) \right\| : \: g \in S_{pq}^r F(\R^d), \: g|_{[0,1)^d} = f \right\}. \]
The spaces $S_{pq}^r B(\R^d), \, S_{pq}^r F(\R^d), \, S_{pq}^r B([0,1)^d)$ and $S_{pq}^r F([0,1)^d)$ are quasi-Banach spaces. For $1 < p < \infty$ we define the Sobolev space with dominating mixed smoothness as
\[ S_p^r H([0,1)^d) = S_{p2}^r F([0,1)^d). \]
If $r \in \N_0$ then it is denoted by $S_p^r W ([0,1)^d)$ and is called classical Sobolev space with dominating mixed smoothness. An equivalent norm for $S_p^r W([0,1)^d)$ is
\[ \sum_{\alpha \in \N_0^d; \, 0 \leq \alpha_i \leq r} \left\| D^{\alpha} f | L_p([0,1)^d) \right\|. \]
Also we have
\[ S_p^0 H([0,1)^d) = L_p([0,1)^d). \]

In \cite{Hi10} Hinrichs analyzed the norm of the discrepancy function of point sets of the Hammersley type in Besov spaces with dominating mixed smoothness. He proved upper bounds which closed the gap for Triebel's results of discrepancy in $S_{pq}^r B([0,1)^2)$-spaces. The result from \cite{Hi10} is that for $1 \leq p,q \leq \infty$ and $0 \leq r < \frac{1}{p}$ there is a constant $c > 0$ such that for any $N \geq 2$, there exists a point set $\P$ in $[0,1)^2$ with $N$ points such that
\[ \left\| D_{\P} | S_{pq}^r B([0,1)^2) \right\| \leq c \, N^{r-1} (\log N)^{\frac{1}{q}}. \]
In \cite{M12a} we proved these bounds for generalized Hammersley type point sets in Besov, Triebel-Lizorkin and Sobolev spaces with domintating mixed smoothness. In this note we close the gap of Triebel's result in arbitrary dimension. We use the same constructions which were used by Chen and Skriganov in \cite{CS02} to prove upper bounds for $L_2$-discrepancy. The notation will mostly orientate on \cite{DP10}. The main result of this note is

\begin{thm} \label{thm_main}

Let $1 \leq p,q \leq \infty$ and $0 < r < \frac{1}{p}$. Then there exists a constant $c > 0$ such that, for any $N \geq 2$, there exists a point set $\P \in [0,1)^d$ with $N$ points such that
\[ \left\| D_{\P} | S_{pq}^r B([0,1)^d) \right\| \leq c \, N^{r-1} \, (\log N)^{\frac{d-1}{q}}. \]

\end{thm}

The point sets in the theorem are the Chen-Skriganov point sets. It was conjectured in \cite{Hi10} that they might satisfy the desired upper bound. The restrictions for the parameter $r$ are necessary. The upper bound $r < \frac{1}{p}$ is due to the fact that we need characteristic functions of intervals to belong to $S_{pq}^r B ([0,1)^d)$ and the condition given by \cite[Theorem 6.3]{T10a}. The restriction $r \geq 0$ comes from the point sets. Anyway, there is a restriction of $r > \frac{1}{p} - 1$ from the fact that we require $S_{pq}^r B ([0,1)^d)$ to have a $b$-adic Haar basis. We have an additional restriction $r > 0$ which is due to our estimations which might not be optimal.

In \cite[Remark 6.28]{T10a} and \cite[Proposition 2.3.7]{Hn10} we find the following very practical embeddings. For $0 < p < \infty, \, 0 < q \leq \infty$ and $r \in \R$ we have
\[ S_{p, \min(p,q)}^r B([0,1)^d) \hookrightarrow S_{pq}^r F([0,1)^d) \hookrightarrow S_{p, \max(p,q)}^r B([0,1)^d). \]
For $0 < p_2 \leq q \leq p_1 < \infty$ and $r \in \R$ we have
\[ S_{p_1 q}^r F([0,1)^d) \hookrightarrow S_{qq}^r B([0,1)^d) \hookrightarrow S_{p_2 q}^r F([0,1)^d). \]
Therefore, we can conclude results for the Triebel-Lizorkin and Sobolev spaces.

\begin{thm}

Let $1 \leq p,q \leq \infty$ and $0 < r < \frac{1}{\max(p,q)}$. Then there exist constants $c_1, c_2 > 0$ such that, for any $N \geq 2$, the discrepancy function of any point set $\mathcal{P}$ in $[0,1)^d$ with $N$ points satisfies
\[ \left\| D_{\mathcal{P}} | S_{pq}^r F([0,1)^d) \right\| \geq c_1 \, N^{r-1} (\log N)^{\frac{d-1}{q}}, \]
and, there exists a point set $\P \in [0,1)^d$ with $N$ points such that
\[ \left\| D_{\P} | S_{pq}^r F([0,1)^d) \right\| \leq c_2 \, N^{r-1} \, (\log N)^{\frac{d-1}{q}}. \]

\end{thm}

\begin{cor}

Let $1 \leq p \leq \infty$ and $0 < r < \frac{1}{\max(p,2)}$. Then there exist constants $c_1, c_2 > 0$ such that, for any $N \geq 2$, the discrepancy function of any point set $\mathcal{P}$ in $[0,1)^d$ with $N$ points satisfies
\[ \left\| D_{\mathcal{P}} | S_p^r H([0,1)^d) \right\| \geq c_1 \, N^{r-1} (\log N)^{\frac{d-1}{2}}, \]
and, there exists a point set $\P \in [0,1)^d$ with $N$ points such that
\[ \left\| D_{\P} | S_p^r H([0,1)^d) \right\| \leq c_2 \, N^{r-1} \, (\log N)^{\frac{d-1}{2}}. \]

\end{cor}

The distribution of points in a cube is not just a theoretical concept. Its application in quasi-Monte Carlo methods is very important. Quadrature formulas need very well distributed point sets. The connection of discrepancy and the error of quadrature formulas can be given for a lot of norms. In \cite[Theorem 6.11]{T10a} Triebel gave this connection for Besov spaces with dominating mixed smoothness. Using the embeddings we get additional results for the Triebel-Lizorkin spaces with dominating mixed smoothness. We define the error of the quadrature formulas in some Banach space $M([0,1)^d)$ of functions on $[0,1)^d$ with $N$ points as
\begin{align*}
\err_N(M([0,1)^d)) = \inf_{\{ x_1, \ldots, x_N \} \subset [0,1)^d} \sup_{f \in M^1_0([0,1)^d)} \left| \int_{[0,1)^d} f(x) \, \dint x - \frac{1}{N}\sum_{i = 1}^N f(x_i) \right|
\end{align*}
where by $M^1_0([0,1)^d)$ we mean the subset of the unit ball of $M([0,1)^d)$ with the property that for all $f \in M^1_0([0,1)^d)$ its extension to $[0,1]^d$ vanishes whenever one of the coordinates of the argument is $1$. Then \cite[Theorem 6.11]{T10a} states that for $1 \leq p,q \leq \infty$ and $\frac{1}{p} < r < \frac{1}{p} + 1$ there exist constants $c_1,c_2 > 0$ such that for every integer $N \geq 2$ we have
\begin{multline*}
c_1 \, \inf_{\mathcal{P} \subset [0,1)^d; \, \#\mathcal{P} = N} \left\| D_{\mathcal{P}} | S_{pq}^r B([0,1)^d) \right\| \\
\leq \err_N(S_{p'q'}^{1-r} B([0,1)^d)) \\
\leq c_2 \, \inf_{\mathcal{P} \subset [0,1)^d; \, \#\mathcal{P} = N} \left\| D_{\mathcal{P}} | S_{pq}^r B([0,1)^d) \right\|
\end{multline*}
where
\[ \frac{1}{p} + \frac{1}{p'} = \frac{1}{q} + \frac{1}{q'} = 1. \]
We can thereby conclude results for the integration errors. For more details we refer to \cite{M12b}.

\begin{thm}

\mbox{}
\begin{enumerate}[(i)]
	\item Let $1 \leq p,q \leq \infty$ and $q < \infty$ if $p = 1$ and $q > 1$ if $p = \infty$. Let $\frac{1}{p} < r < 1$. Then there exist constants $c_1, C_1 > 0$ such that, for any integer $N \geq 2$, we have
	\[ c_1 \, \frac{(\log N)^{\frac{(q-1)(d-1)}{q}}}{N^r} \leq \err_N(S_{pq}^r B) \leq C_1 \, \frac{(\log N)^{\frac{(q-1)(d-1)}{q}}}{N^r}. \]
	\item Let $1 \leq p,q < \infty$. Let $\frac{1}{\min(p,q)} < r < 1$. Then there exist constants $c_2, C_2 > 0$ such that, for any integer $N \geq 2$, we have
	\[ c_2 \, \frac{(\log N)^{\frac{(q-1)(d-1)}{q}}}{N^r} \leq \err_N(S_{pq}^r F) \leq C_2 \, \frac{(\log N)^{\frac{(q-1)(d-1)}{q}}}{N^r}. \]
	\item Let $1 \leq p < \infty$. Let $\frac{1}{\min(p,2)} < r < 1$. Then there exist constants $c_3, C_3 > 0$ such that, for any integer $N \geq 2$, we have
	\[ c_3 \, \frac{(\log N)^{\frac{d-1}{2}}}{N^r} \leq \err_N(S_p^r H) \leq C_3 \, \frac{(\log N)^{\frac{d-1}{2}}}{N^r}. \] \label{BTY_quote}
\end{enumerate}

\end{thm}

In the next sections we introduce the constructions by Chen and Skriganov which we will use to prove our result. In order to do so we will calculate $b$-adic Haar coefficients of the discrepancy function. We also will introduce $b$-adic Walsh functions.

\section{The $b$-adic Haar bases}
For some integer $b\geq 2$ a $b$-adic interval of length $b^{-j},\, j\in\N_0$ in $[0,1)$ is an interval of the form
\[ I_{jm} = \big[ b^{-j} m, b^{-j} (m+1) \big)\]
for $m=0,1,\ldots,b^j-1$. For $j \in \N_0$ we divide $I_{jm}$ into $b$ intervals of length $b^{-j - 1}$, i.e. $I_{jm}^k = I_{j + 1,bm + k},\, k=0,\ldots,b - 1$. As an additional notation we put $I_{-1,0}^{-1} = I_{-1,0} = [0,1)$. Let $\D_j = \{0,1,\ldots,b^j-1\}$ and $\B_j = \{1,\ldots,b-1\}$ for $j \in \N_0$ and $\D_{-1} = \{0\}$ and $\B_{-1} = \{1\}$. The $b$-adic Haar functions $h_{jml}$ have support in $I_{jm}$. For any $j \in \N_0,\, m \in \D_j,\, l \in \B_j$ and any $k=0,\ldots,b-1$ the value of $h_{jml}$ in $I_{jm}^k$ is $e^{\frac{2\pi i}{b}lk}$. We denote the indicator function of $I_{-1,0}$ by $h_{-1,0,1}$. Let $\N_{-1} = \{-1,0,1,2,\ldots\}$. The functions $h_{jml},\, j \in \N_{-1},\,m \in \D_j,\,l\in\B_j$ are called $b$-adic Haar system. Normalized in $L_2([0,1))$ we obtain the orthonormal $b$-adic Haar basis of $L_2([0,1))$. A proof of this fact can be found in \cite{RW98}.

For $j = (j_1, \dots, j_d) \in \N_{-1}^d$, $m = (m_1, \ldots, m_d) \in \D_j = \D_{j_1} \times \ldots \times \D_{j_d}$ and $l = (l_1, \ldots, l_d) \in \B_j = \B_{j_1} \times \ldots \times \B_{j_d}$, the Haar function $h_{jml}$ is given as the tensor product $h_{jml}(x) = h_{j_1 m_1 l_1}(x_1) \ldots h_{j_d m_d l_d}(x_d)$ for $x = (x_1, \ldots, x_d) \in [0,1)^d$. We will call $I_{jm} = I_{j_1 m_1} \times \ldots \times I_{j_d m_d}$ $b$-adic boxes. For $k = (k_1,\ldots,k_d)$ where $k_i \in \{ 0,\ldots,b - 1 \}$ for $j_i \in \N_0$ and $k_i = -1$ for $j_i = -1$ we put $I_{jm}^k = I_{j_1 m_1}^{k_1} \times \ldots \times I_{j_d m_d}^{k_d}$. The functions $h_{jml},\, j \in \N_{-1}^d,\,m \in \D_j,\,l\in\B_j$ are called $d$-dimensional $b$-adic Haar system. Normalized in $L_2([0,1)^d)$ we obtain the orthonormal $b$-adic Haar basis of $L_2([0,1)^d)$.

For any function $f \in L_2([0,1)^d)$ we have by Parseval's equation
\begin{align}
\|f|L_2\|^2 = \sum_{j\in\N_{-1}^d} b^{\max(0,j_1) + \ldots + \max(0,j_d)} \sum_{m\in \D_j, \, l\in\B_j} |\mu_{jml}|^2
\end{align}
where
\begin{align}
\mu_{jml} = \mu_{jml}(f) = \int_{[0,1)^d} f(x) h_{jml}(x) \, \dint x
\end{align}
are the $b$-adic Haar coefficients of $f$.

In \cite{M12a} we gave an equivalent norm for the Besov spaces with dominating mixed smoothness. Let $j = (j_1, \ldots, j_d) \in \N_{-1}^d$ and let $s = \#\{ i = 1, \ldots, d: \, j_i \neq -1 \}$. Then we can choose a subsequence $(\eta_\nu)_{\nu = 1}^s$ of $(1, \ldots, d)$ such that for all $\nu = 1, \ldots, s$ we have $j_{\eta_{\nu}} \neq -1$ while all other $j_i$ are equal to $-1$. Then we have
\begin{align} \label{eq_quasinorm}
\left\| f | S_{pq}^r B([0,1)^d) \right\| \approx \left( \sum_{j \in \N_{-1}^d} b^{(j_{\eta_1} + \ldots + j_{\eta_s})(r - \frac{1}{p} + 1) q} \left( \sum_{m \in \D_j, \, l \in \B_j} | \mu_{jml}|^p \right)^{\frac{q}{p}} \right)^{\frac{1}{q}}
\end{align}
for every $f \in S_{pq}^r B([0,1)^d)$.

\section{Constructions of Chen and Skriganov}
In this section we will describe the constructions of point sets proposed by Chen and Skriganov. We describe the constructions for those $N = b^n$ where $n$ is divisible by $2d$, i.e. $n =2dw$ for some $w \in \N$. Point sets with arbitrary number of points can be constructed with the usual method as is described for example in \cite[Section 16.1]{DP10}. We give some notation and definitions.

We begin with the definition of digital nets. Let $d \geq 1$, $b \geq 2$, and $n \geq 0$ be integers. A point set $\P \in [0,1)^d$ of $b^n$ points is called a net in base $b$ if every $b$-adic interval with volume $b^{-n}$ contains exactly one point of $\P$. Usually the term $(0,n,d)$-net is used but here we just say net for shortness. Digital nets in base $b$ are special nets in base $b$ that can be constructed using $n \times n$ matrices $C_1, \ldots, C_d$. We describe the digital method to construct digital nets. Let $b$ be a prime number. Let $n \in \N_0$. Let $C_1, \ldots, C_d$ be $n \times n$ matrices with entries from $\F_b$. We generate the net point $x_r = (x_r^1, \ldots, x_r^d)$ with $0 \leq r < b^n$. We expand $r$ in base $b$ as
\[ r = r_0 + r_1 b + \ldots + r_{n-1}b^{n-1} \]
with digits $r_k \in \{ 0, 1, \ldots, b-1 \}$, $1 \leq k \leq n-1$. We put $\bar{r} = (r_0, \ldots, r_{n-1})^{\top} \in \F_b^n$ and $\bar{h}_{r,i} = C_i \, \bar{r} = (h_{r,i,1}, \ldots, h_{r,i,n})^{\top} \in \F_b$, $1 \leq i \leq d$. Then we get $x_r^i$ as
\[ x_r^i = \frac{h_{r,i,1}}{b} + \ldots + \frac{h_{r,i,n}}{b^n}. \]
A point set $\left\{ x_0, \ldots x_{b^n-1} \right\}$ constructed with the digital method is called digital net in base $b$ with generating matrices $C_1, \ldots C_d$ if it is a net in base $b$.

For a digital net $\P$ with generating matrices $C_1,\ldots,C_d$ over $\F_b$ we define
\[ \Dn'(C_1,\ldots,C_d) = \left\{ t \in \{ 0, \ldots, b^n - 1 \}^d \, : \, C_1^{\top} \, \bar{t}_1 + \ldots + C_d^{\top} \, \bar{t}_d = 0 \right\} \backslash \{ 0 \} \]
where $t = (t_1, \ldots, t_d)$ and for $1 \leq i \leq d$ we denote by $\bar{t}_i$ the $n$-dimensional column vector of $b$-adic digits of $t_i$. Instead of $(0, \ldots, 0)$ we just wrote $0$.

Now let $\alpha \in \N$ with $b$-adic expansion $\alpha = \alpha_0 + \alpha_1 b + \ldots + \alpha_{h - 1} b^{h - 1}$ where $\alpha_{h - 1} \neq 0$. We define the Niederreiter-Rosenbloom-Tsfasman (NRT) weight by $\varrho(\alpha) = h$. Furthermore, we define $\varrho(0) = 0$. We define the Hamming weight $\varkappa(\alpha)$ as the number of non-zero digits $\alpha_{\nu}$, $0 \leq \nu < \varrho(\alpha)$.

For $\alpha = (\alpha_1,\ldots,\alpha_d) \in \N_0^d$, let
\[ \varrho^d(\alpha) = \sum_{i = 1}^d \varrho(\alpha_i) \text{ and } \varkappa^d(\alpha) =  \sum_{i = 1}^d \varkappa(\alpha_i). \]
Now let $b$ be prime and let $n \in \N$. Let $\C$ be some $\F_b$-linear subspace of $\F_b^{dn}$. Then we define the dual space $\C^{\perp}$ relative to the standard inner product in $\F_b^{dn}$ as
\[ \C^{\perp} = \left\{ A \in \F_b^{dn}: \, B \cdot A = 0 \text{ for all } B \in \C \right\}. \]
We have $\dim(\C^{\perp}) = dn - \dim(\C)$ and $(\C^{\perp})^{\perp} = \C$.

Now let $a = (a_1, \ldots, a_n) \in \F_b^n$ and
\[ v_n(a) = \begin{cases} 0 & \text{ if } a = 0, \\ \max\left\{ \nu: \, a_{\nu} \neq 0 \right\} & \text{ if } a \neq 0. \end{cases} \]
We define $\varkappa_n(a)$ as the number of indices $1 \leq \nu \leq n$ such that $a_{\nu} \neq 0$. Let $A = (a_1, \ldots, a_d) \in \F_b^{dn}$ with $a_i \in \F_b^n$ for $1 \leq i \leq d$ and let
\[ v_n^d(A) = \sum_{i = 1}^d v_n(a_i) \text{ and } \varkappa_n^d(A) = \sum_{i = 1}^d \varkappa_n(a_i). \]
Now let $\C \neq \left\{ 0 \right\}$. Then the minimum distance of $\C$ is defined as
\[ \delta_n(\C) = \min\left\{ V_n(A): \, A \in \C \textbackslash \left\{ 0 \right\} \right\}. \]
Furthermore, let $\delta_n(\left\{ 0 \right\}) = dn + 1$. Finally we define
\[ \varkappa_n(\C) = \min\left\{ \varkappa_n(A): \, A \in \C \textbackslash \left\{ 0 \right\} \right\}. \]
The weight $\varkappa_n(\C)$ is called Hamming weight of $\C$.

Now we continue with the constructions. So let $w \in \N$ such that $n = 2dw$.
Let
\[ f(z) = f_0 + f_1 z + \ldots + f_{h - 1} z^{h - 1} \]
be a polynomial in $\F_b[z]$. For every $\lambda \in \N$ the $\lambda$-th hyper-derivative is
\[ \partial^{\lambda} f(z) = \sum_{i=0}^{h - 1} \binom{i}{\lambda} f_{\lambda} z^{i - \lambda}. \]
We use the usual convention for the binomial coefficient modulo $b$ that $\binom{i}{\lambda} = 0$ whenever $\lambda > i$.
Let $b \geq 2d^2$ be a prime. Then there are $2d^2$ distinct elements $\beta_{i,\nu} \in \F_b$ for $1 \leq i \leq d$ and $1 \leq \nu \leq 2d$. For $1 \leq i \leq d$ let
\[ a_i(f) = \left( \left( \partial^{\lambda - 1} f(\beta_{i,\nu}) \right)_{\lambda = 1}^w \right)_{\nu=1}^{2d} \in \F_b^n. \]
We define $\C_n \subset \F_b^{dn}$ as
\[ \C_n = \left\{ A(f) = (a_1(f), \ldots, a_d(f)): \, f \in \F_b[z], \, \dg(f) < n \right\}. \]
Clearly, $\C_n$ has exactly $b^n$ elements. The set of polynomials in $\F_b[z]$ with $\dg(f) < n$ is closed under addition and scalar multiplication over $\F_b$, hence $\C_n$ is an $\F_b$-linear subspace of $\F_b^{dn}$. For example from \cite[Theorem 16.28]{DP10} one learns that $\C_n$ has dimension $n$ while its dual $nd - n$ and it satisfies
\begin{align} \label{dualspace}
\varkappa_n(\C_n^{\perp}) \geq 2d + 1 \text{ and } \delta_n(\C_n^{\perp}) \geq n + 1.
\end{align}
Finally we only need to transfer $\C_n$ into the unit cube $[0,1)^d$ as a point set. To do so we define a mapping $\Phi_n: \, \F_b^{dn} \rightarrow [0,1)^d$. Let $a = (a_1, \ldots, a_n) \in \F_b^n$, we set
\[ \Phi_n(a) = \frac{a_1}{b} + \ldots + \frac{a_n}{b^n} \]
and for $A = (a_1, \ldots, a_d) \in F_b^{dn}$, we set
\[ \Phi_n^d(A) = \left( \Phi_n(a_1), \ldots, \Phi_n(a_d) \right). \]

So we are ready to define the point set that proves our main result. The point set of Chen and Skriganov is given by
\[ \CS_n = \Phi_n^d(\C_n) \]
and contains exactly $N = b^n$ points.
From \cite[Theorem 7.14]{DP10} we finally learn that $\CS_n$ is a digital net in base $b$ since for every $\F_b$-linear subspace $\C$ of $\F_b^{dn}$ with dimension $n$ and dual space satisfying \eqref{dualspace}, $\Phi_n^d(\C)$ is a digital net in base $b$ with some generating matrices $C_1, \ldots, C_d$. We will call $\Phi_n^d(\C)$ the corresponding digital net. As a final remark of this section we would like to note that we needed $b$ to be large so that $\F_b$ has enough distinct elements. But there are even general rules for nets on the minimum of $b$ such that, a net in base $b$ can exist (see \cite[Chapter 4]{DP10}.

\section{The $b$-adic Walsh functions}
Let $b \geq 2$ be an integer. For some $\alpha \in \N_0$ with $b$-adic expansion $\alpha = \alpha_0 + \alpha_1 b + \ldots + \alpha_{\varrho(\alpha) - 1} b^{\varrho(\alpha) - 1}$ we define the $\alpha$-th $b$-adic Walsh function $\wal_{\alpha}: \, [0,1) \rightarrow \Cx$, as
\[ \wal_{\alpha}(x) = e^{\frac{2 \pi i}{b} (\alpha_0 x_1 + \alpha_1 x_2 + \ldots + \alpha_{\varrho(\alpha) - 1} x_{\varrho(\alpha)})}, \]
for $x \in [0,1)$ with $b$-adic expansion  $x = x_1 b^{-1} + x_2 b^{-2} + \ldots$. The functions $\wal_{\alpha}, \, \alpha \in \N_0$ are called $b$-adic Walsh system.

For $\alpha = (\alpha^1, \ldots, \alpha^d) \in \N_0^d$ the Walsh function $\wal_{\alpha}$ is given as the tensor product $\wal_{\alpha}(x) = \wal_{\alpha^1}(x_1) \ldots \wal_{\alpha^d}(x_d)$ for $x = (x_1, \ldots, x_d) \in [0,1)^d$. The functions $\wal_{\alpha}, \, \alpha \in \N_0^d$ are called $d$-dimensional $b$-adic Walsh system.

For $\alpha \in \N$ the function $\wal_{\alpha}$ is constant on $b$-adic intervals $I_{\varrho(\alpha)m}$ for any $m \in \D_{\varrho(\alpha)}$. Further, $\wal_0$ is constant on $[0,1)$ with value $1$. We have
\[ \int_{[0,1)} \wal_{\alpha}(x)\dint x = \begin{cases} 1 & \text{ if } \alpha = 0, \\ 0 & \text{ if } \alpha \neq 0, \end{cases} \]
and for $\alpha, \beta \in \N_0^d$ we have
\[ \int_{[0,1)^d} \wal_{\alpha}(x) \overline{\wal_{\beta}(x)} \dint x = \begin{cases} 1 & \text{ if } \alpha = \beta, \\ 0 & \text{ if } \alpha \neq \beta. \end{cases} \]
The $d$-dimensional $b$-adic Walsh system is an orthonormal basis in $L_2([0,1)^d)$. The proofs of these facts can be found for example in Appendix A of \cite{DP10}.

\section{Calculation of the $b$-adic Haar coefficients}
Before we can compute the Haar coefficients we need some easy lemmas. We omit the proofs since they are nothing further but easy exercises.

\begin{lem} \label{lem_haar_coeff_besov_x}

Let $f(x) = x_1 \cdot \ldots \cdot x_d$ for $x=(x_1,\ldots,x_d) \in [0,1)^d$. Let $j \in \N_{-1}^d, \, m \in \D_j, l \in \B_j$ and let $\mu_{jml}$ be the $b$-adic Haar coefficient of $f$. Then
\[ \mu_{jml} = \frac{b^{-2j_{\eta_1} - \ldots - 2j_{\eta_s} - s}}{2^{d-s}(e^{\frac{2\pi i}{b} l_{\eta_1}} - 1) \cdot \ldots \cdot (e^{\frac{2\pi i}{b} l_{\eta_s}} - 1)}. \]

\end{lem}

\begin{lem} \label{lem_haar_coeff_besov_indicator}

Let $z = (z_1,\ldots,z_d) \in [0,1)^d$ and $g(x) = \chi_{[0,x)}(z)$ for $x = (x_1, \ldots, x_d) \in [0,1)^d$. Let $j \in \N_{-1}^d, \, m \in \D_j, l \in \B_j$ and let $\mu_{jml}$ be the $b$-adic Haar coefficient of $g$. Then $\mu_{jml} = 0$ whenever $z$ is not contained in the interior of the $b$-adic box $I_{jm}$ supporting the functions $h_{jml}$. If $z$ is contained in the interior of $I_{jm}$ then there is a $k = (k_1,\ldots,k_d)$ with $k_i \in \{0,1,\ldots,b-1\}$ if $j_i \neq -1$ or $k_i = -1$ if $j_i = -1$ such that $z$ is contained in $I_{jm}^k$. Then 
	      \begin{multline*}
        \mu_{jml} = b^{-j_{\eta_1} - \ldots - j_{\eta_s} - s} \prod_{1 \leq i \leq d; \, j_i = -1}(1-z_i) \times\\
        \times \prod_{\nu = 1}^s \left[ (bm_{\eta_{\nu}}+k_{\eta_{\nu}}+1-b^{j_{\eta_{\nu}}+1}z_{\eta_{\nu}}) e^{\frac{2\pi i}{b}k_{\eta_{\nu}} l_{\eta_{\nu}}} + \sum_{r_{\eta_{\nu}} = k_{\eta_{\nu}}+1}^{b-1} e^{\frac{2\pi i}{b}r_{\eta_{\nu}} l_{\eta_{\nu}}} \right].
        \end{multline*}     
\end{lem}

\begin{lem} \label{lem_lambda_s_minus_1}

Let $\lambda \in \N_0$ and $s \in \N$. Then
\[ \# \left\{ (j_1, \ldots, j_s) \in \N_0^s: \, j_1 + \ldots + j_s = \lambda \right\} \leq (\lambda + 1)^{s-1}. \]

\end{lem}

We consider the Walsh series expansion of the function $\chi_{[0,y)}$
\begin{align}
\chi_{[0,y)}(x) = \sum_{t = 0}^\infty \hat{\chi}_{[0,y)}(t) \wal_t(x),
\end{align}
where for $t \in \N_0$ with $b$-adic expansion $t = \tau_0 + \tau_1 b + \ldots + \tau_{\varrho(t) - 1} b^{\varrho(t) - 1}$, the $t$-th Walsh coefficient is given by
\[ \hat{\chi}_{[0,y)}(t) = \int_0^1 \chi_{[0,y)}(x) \overline{\wal_t(x)} \dint x = \int_0^y \overline{\wal_t(x)} \dint x. \]
For $t > 0$ we put $t = t' + \tau_{\varrho(t) - 1} b^{\varrho(t) - 1}$.

The following is called Fine-Price formulas and was first proved in \cite{F49} (dyadic case) and \cite{P57} ($b$-adic version). One often finds it in literature, e.g. see \cite[Lemma 14.8]{DP10} for an easy understandable proof.

\begin{lem} \label{chi_roof_lem}

Let $b \geq 2$ be an integer and $x \in [0,1)$. Then we have
\[ \hat{\chi}_{[0,y)}(0) = y = \frac{1}{2} + \sum_{a = 1}^\infty \sum_{z = 1}^{b-1} \frac{1}{b^a (e^{-\frac{2 \pi i}{b} z} - 1)} \wal_{z b^{a-1}}(y) \]
and for any integer $t > 0$ we have
\begin{multline*}
\hat{\chi}_{[0,y)}(t) = \frac{1}{b^{\varrho(t)}}\left( \frac{1}{1 - \e^{-\frac{2 \pi \im}{b} \tau_{\varrho(t) - 1}}} \overline{\wal_{t'}(y)} \right. +\\
+ \left( \frac{1}{\e^{-\frac{2 \pi \im}{b} \tau_{\varrho(t) - 1}} - 1} + \frac{1}{2} \right) \overline{\wal_t(y)} +\\
+ \left. \sum_{a = 1}^\infty \sum_{z = 1}^{b-1} \frac{1}{b^a (\e^{\frac{2 \pi \im}{b} z} - 1)} \overline{\wal_{z b^{\varrho(t)+a-1} + t}(y)} \right).
\end{multline*}

\end{lem}

The first part of the Lemma is \cite[Lemma A.22]{DP10}, the second is \cite[Lemma 14.8]{DP10}.

Let $n \in \N_0$, we consider the approximation of $\chi_{[0,y)}$ by the truncated series
\begin{align}
\chi_{[0,y)}^{(n)}(x) = \sum_{t = 0}^{b^n - 1} \hat{\chi}_{[0,y)}(t) \wal_t(x).
\end{align}
Let $N$ be a positive integer. Then we put for some point set $\P$ in $[0,1)^d$ with $N$ points
\begin{align}
\Theta_{\P}(y) = \frac{1}{N} \sum_{z \in \P} \chi_{[0,y)}^{(n)}(z) - y_1 \cdot \ldots \cdot y_d.
\end{align}
Let
\begin{align} \label{split}
D_{\P}(y) = \Theta_{\P}(y) + R_{\P}(y).
\end{align}
We now restrict ourselves again to the case where $b$ is prime.

\begin{lem} \label{lem_duality_into_disc}

Let $\{ x_0, \ldots, x_{b^n - 1} \}$ be a digital net in base $b$ generated by the matrices $C_1, \ldots, C_d$. Then for $t \in \{ 0, \ldots, b^n - 1 \}^d$, we have
\[ \sum_{h = 0}^{b^n - 1} \wal_t(x_h) = \begin{cases} b^n & \text{ if } t \in \Dn(C_1, \ldots, C_d), \\ 0 & \text{ otherwise}. \end{cases} \]

\end{lem}
The proof of this fact can be found in \cite[Section 4.4]{DP10}

\begin{lem} \label{theta_lem}

Let $\C$ be an $\F_b$-linear subspace of $\F_b^{dn}$ of dimension $n$ and let $\P = \Phi_n^d(\C)$ denote the corresponding digital net in base $b$ with generating matrices $C_1, \ldots, C_d$. Then
\[ \Theta_{\P}(y) = \sum_{t \in \Dn'(C_1, \ldots, C_d)} \hat{\chi}_{[0,y)}(t). \]

\end{lem}
The proof of this fact is contained in the proof of \cite[Lemma 16.22]{DP10}.

\begin{lem} \label{rest_lem}

There exists a constant $c > 0$ such that, for any $n \in \N_0$ and for any $\F_b$-linear subspace $\C$ of $\F_b^{dn}$ of dimension $n$ with dual space $\C^{\perp}$ satisfying $\delta_n(\C^{\perp}) \geq n + 1$ with the corresponding digital net $\P = \Phi_n(\C)$ and for every $y \in [0,1)^d$, we have
\[ |R_{\P}(y)| \leq c \, b^{-n}. \]

\end{lem}

For a proof of this lemma the interested reader is referred to \cite[Lemma 16.21]{DP10}.

We introduce a very common notation. For functions $f, g \in L_2(\Q^d)$ we write
\[ \left\langle f, g \right\rangle = \int_{\Q^d} f \, \bar{g}. \]

\begin{prp} \label{prp_min1}

Let $j = (-1, \ldots, -1), \, m = (0, \ldots, 0), \, l = (1, \ldots, 1)$. Then there exists a constant $c > 0$ independent of $n$ such that
\[ |\mu_{jml}(D_{\CS_n})| \leq c \, b^{-n}. \]

\end{prp}

\begin{proof}
As in \eqref{split} we split $D_{\CS_n}(y) = \Theta_{\CS_n}(y) + R_{\CS_n}(y)$ and we know from Lemma \ref{rest_lem} (since $\CS_n$ is a digital net) that there is a constant $c > 0$ such that $|R_{\CS_n}(y)| \leq c \, b^{-n}$. Using Lemma \ref{theta_lem} we can calculate the Haar coefficient
\[ \mu_{jml}(D_{\CS_n}) = \left\langle \Theta_{\CS_n} + R_{\CS_n} , h_{jml} \right\rangle. \]
To do so we use the fact that $h_{jml} = \wal_{(0, \ldots, 0)}$. Now we consider the one-dimensional case first and from the first part of Lemma \ref{chi_roof_lem} we get
\[ \left\langle \hat{\chi}_{[0,\cdot)}(0),\wal_0 \right\rangle = \frac{1}{2}. \]
Now let $t > 0$. Then by the second part of Lemma \ref{chi_roof_lem} we have
\[ \left\langle \hat{\chi}_{[0,\cdot)}(t),\wal_0 \right\rangle = \begin{cases}
\frac{1}{b^{\varrho(t)}} \frac{1}{1 - e^{-\frac{2 \pi i}{b} \tau_{\varrho(t)-1}}} & t' = 0, \\
0 & t' \neq 0.
\end{cases} \]
This means that we can find a constant $c_1 > 0$ such that for every integer $t \geq 0$ we have
\[ \left| \left\langle \hat{\chi}_{[0,\cdot)}(t),\wal_0 \right\rangle \right| \leq c_1 \, b^{-\varrho(t)} \]
and
\[ \left\langle \hat{\chi}_{[0,\cdot)}(t),\wal_0 \right\rangle = 0 \]
if $t > 0$ and $t' \neq 0$.

Now suppose, we have some $t \in \Dn'(C_1, \ldots, C_d)$ such that
\[ \left\langle \hat{\chi}_{[0,\cdot)}(t), \wal_{(0, \ldots, 0)} \right\rangle \neq 0. \]
Then for all $1 \leq i \leq d$ we have
\[ \left\langle \hat{\chi}_{[0,\cdot)}(t_i), \wal_0 \right\rangle \neq 0. \]
Then necessarily $t_i = \tau_i b^{\varrho(t_i)-1}$ (since $t_i' = 0$) or $t_i = 0$ for any $i = 1, \ldots, d$ which means that either $\varkappa(t_i) = 1$ or $\varkappa(t_i) = 0$. In any case we have $\varkappa^d(t) \leq d$ which contradicts to $\varkappa_n(\mathcal{C}_n^{\bot}) \geq 2d + 1$ as must be the case according to above. Therefore, for all $t \in \Dn'(C_1, \ldots, C_d)$ we have
\[ \left\langle \hat{\chi}_{[0,\cdot)}(t), \wal_{(0, \ldots, 0)} \right\rangle = 0 \]
and from Lemma \ref{theta_lem} follows $\left\langle \Theta_{\CS_n}, \wal_0 \right\rangle = 0$.
Hence we have
\[ |\mu_{jml}(D_{\CS_n})| \leq |\left\langle \Theta_{\CS_n} , \wal_0 \right\rangle| + |\left\langle R_{\CS_n},\wal_0 \right\rangle| \leq c \, b^{-n}. \]

\end{proof}

\begin{lem} \label{lem_scalprod_1}

Let $j \in \N_{-1}$, $m \in \D_j$, $l \in \B_j$ and $\alpha \in \N_0$. Then
\begin{enumerate}[(i)]
  \item if $j \in \N_0$ and $\varrho(\alpha) = j + 1$ and $\alpha_{j} = l$. Then
\[ |\left\langle h_{jml} , \wal_{\alpha} \right\rangle| = b^{-j}, \]
  \item if $j = -1, \, m = l = 0$ and $\alpha = 0$ then
\[ \left\langle h_{jml} , \wal_{\alpha} \right\rangle| = 1, \]
  \item if $\varrho(\alpha) \neq j + 1$ or $\alpha_{j} \neq l$ then
\[ |\left\langle h_{jml} , \wal_{\alpha} \right\rangle| = 0. \]
\end{enumerate}

\end{lem}

\begin{proof}
The second claim and the third for $j = -1$ are trivial so let $j \geq 0$. Let $y \in [0,1)$. We expand $\alpha$ and $y$ as
\[ \alpha = \alpha_0 + \alpha_1 b + \ldots + \alpha_{\varrho(\alpha)-1} b^{\varrho(\alpha)-1} \]
and
\[ y = y_1 b^{-1} + y_2 b^{-2} + \ldots. \]
Hence
\[ \wal_{\alpha} (y) = e^{\frac{2 \pi i}{b} (\alpha_0 y_1 + \ldots + \alpha_{\varrho(\alpha)-1} y_{\varrho(\alpha)})}. \]
The function $\wal_{\alpha}$ is constant on the intervals
\[ \left[\right.b^{-\varrho(\alpha)} \delta , b^{-\varrho(\alpha)} (\delta + 1)\left.\right) \]
for any integer $0 \leq \delta < b^{\varrho(\alpha)}$. The function $h_{jml}$ is constant on the intervals
\[ I_{jm}^k = \left[\right.b^{-j-1} (bm+k) , b^{-j-1} (bm+k+1)\left.\right) \]
for any integer $0 \leq k < b$. Now suppose that either $j+1 > \varrho(\alpha)$ or $j+1 < \varrho(\alpha)$. This would mean that either
\[ I_{jm} = \left[\right. b^{-j} m , b^{-j} (m+1)\left.\right) \subseteq \left[\right.b^{-\varrho(\alpha)} \delta , b^{-\varrho(\alpha)} (\delta + 1)\left. \right) \]
in the first case or
\[ \left[\right. b^{-\varrho(\alpha)} \delta , b^{-\varrho(\alpha)} (\delta + 1)\left. \right) \subset I_{jm}^k \]
for some $k$ in the second case or in both cases
\[ \left[\right. b^{-j} m , b^{-j} (m+1)\left. \right) \cap \left[\right. b^{-\varrho(\alpha)} \delta , b^{-\varrho(\alpha)} (\delta + 1)\left. \right) = \emptyset \]
In any case
\[ \left\langle h_{jml} , \wal_{\alpha} \right\rangle = 0. \]
So the only relevant case is $j+1 = \varrho(\alpha)$. Then either again
\[ \left[\right. b^{-j} m , b^{-j} (m+1)\left.\right) \cap \left[\right.b^{-\varrho(\alpha)} \delta , b^{-\varrho(\alpha)} (\delta + 1)\left. \right) = \emptyset \]
or
\[ \left[\right. b^{-\varrho(\alpha)} \delta , b^{-\varrho(\alpha)} (\delta + 1)\left. \right) = I_{jm}^k \]
for some $k$. We consider the last possibility.
The value of $h_{jml}$ on $I_{jm}^k$ is $e^{\frac{2 \pi i}{b} l k}$. To calculate the value of $\wal_{\alpha}$ we expand $m$ as
\[ m = m_1 + m_2 b + \ldots + m_j b^{j-1}. \]
Clearly, $0 \leq b m + k < b^{j+1}$. Hence,
\[ b^{-j-1} (b m + k) = m_j b^{-1} + \ldots + m_2 b^{-j+1} + m_1 b^{-j} + k b^{-j-1}. \]
So,
\[ \wal_{\alpha} (b^{-j-1} (b m + k)) = e^{\frac{2 \pi i}{b} \alpha_0 m_j + \ldots + \alpha_{j-1} m_1 + \alpha_j k}. \]
Now we can calculate
\begin{align*}
\overline{\left\langle h_{jml} , \wal_{\alpha} \right\rangle} & = \int_{I_{jm}} \overline{h_{jml}(y)} \wal_{\alpha}(y) \dint y\\
& = \sum_{k=0}^{b-1} \int_{I_{jm}^k} \overline{h_{jml}(y)} \wal_{\alpha}(y) \dint y\\
& = b^{-j-1} \sum_{k=0}^{b-1} e^{\frac{2 \pi i}{b} \alpha_0 m_j + \ldots + \alpha_{j-1} m_1 + (\alpha_j-l) k}\\
& = b^{-j-1} e^{\frac{2 \pi i}{b} \alpha_0 m_j + \ldots + \alpha_{j-1} m_1} \sum_{k=0}^{b-1} e^{(\alpha_j-l) k}\\
& = \begin{cases}
b^{-j} e^{\frac{2 \pi i}{b} \alpha_0 m_j + \ldots + \alpha_{j-1} m_1} & \alpha_j = l,\\
0 & \alpha_j \neq l
\end{cases}
\end{align*}
and the lemma follows.

\end{proof}

\begin{lem} \label{lem_scalprod_2}

There exists a constant $c > 0$ with the following property. Let $t, \alpha \in \N_0$. Then if $\alpha = t'$ or $\alpha = t + \tau \, b^{\varrho(t) + a - 1}$ for some integer $0 \leq \tau \leq b - 1$ and $a \geq 1$ then
\[ \left| \left\langle \hat{\chi}_{[0,\cdot)}(t) , \wal_{\alpha} \right\rangle \right| \leq c \, b^{-\max(\varrho(t), \varrho(\alpha))}. \]
If $\alpha \neq t'$ and there are no integers $0 \leq \tau \leq b - 1$ and $a \geq 1$ such that $\alpha = t + \tau \, b^{\varrho(t) +a - 1}$ then
\[ \left\langle \hat{\chi}_{[0,\cdot)}(t) , \wal_{\alpha} \right\rangle = 0. \]

\end{lem}

\begin{proof}
We use Lemma \ref{chi_roof_lem}. First let $t > 0$. Suppose that $\alpha = t'$, so $\varrho(\alpha) < \varrho(t)$. Then
\[ \left| \left\langle \hat{\chi}_{[0,\cdot)}(t) , \wal_{\alpha} \right\rangle \right| = \left| \frac{1}{1 - e^{\frac{-2 \pi i}{b} \tau_{\varrho(t) - 1}}} \right| \, b^{-\varrho(t)} \leq c \, b^{-\varrho(t)}. \]
If $\alpha = t$ meaning that $\varrho(\alpha) = \varrho(t)$ then
\[ \left| \left\langle \hat{\chi}_{[0,\cdot)}(t) , \wal_{\alpha} \right\rangle \right| \leq \left| \frac{1}{e^{\frac{-2 \pi i}{b} \tau_{\varrho(t) - 1}} - 1} + \frac{1}{2} \right| \, b^{-\varrho(t)} \leq c \, b^{-\varrho(t)}. \]
Now let $\alpha = t + \tau \, b^{\varrho(t) +a - 1}$ for some $1 \leq \tau \leq b - 1$ and $a \geq 1$. Hence $\varrho(\alpha) = \varrho(t) + a$. Then
\[ \left| \left\langle \hat{\chi}_{[0,\cdot)}(t) , \wal_{\alpha} \right\rangle \right| = \left| \frac{1}{e^{\frac{2 \pi i}{b} \tau} - 1} \right| \, b^{-\varrho(t)} \, b^{-a} \leq c \, b^{-\varrho(\alpha)}. \]
For any other $\alpha$ clearly,
\[ \left\langle \hat{\chi}_{[0,\cdot)}(t) , \wal_{\alpha} \right\rangle = 0. \]
Now we consider the case $t = 0$. Then for $\alpha = 0$ (meaning $\varrho(\alpha) = 0$) we have
\[ \left| \left\langle \hat{\chi}_{[0,\cdot)}(t) , \wal_{\alpha} \right\rangle \right| = \frac{1}{2} \leq c \, b^{-\varrho(\alpha)}. \]
Let $\alpha = \tau \, b^{a - 1}$ for some $1 \leq \tau \leq b - 1$ and $a \geq 1$. Then $\varrho(\alpha) = a$ and
\[ \left| \left\langle \hat{\chi}_{[0,\cdot)}(t) , \wal_{\alpha} \right\rangle \right| = \left| \frac{1}{e^{\frac{2 \pi i}{b} \tau} - 1} \right| \, b^{-a} \leq c \, b^{-\varrho(\alpha)}. \]
For any other $\alpha$ again clearly,
\[ \left\langle \hat{\chi}_{[0,\cdot)}(t) , \wal_{\alpha} \right\rangle = 0. \]

\end{proof}

We now need an additional notation. For any function $f \, : \, \Fb_b^{dn} \longrightarrow \Cx$ we call $\hat{f}$ given by
\[ \hat{f}(B) = \sum_{A \in \Fb_b^{dn}} \e^{\frac{2 \pi \im}{b} A \cdot B} f(A) \]
for $B \in \Fb_b^{dn}$ the Walsh transform of $f$.

The following two facts can be found in \cite{DP10}. The first lemma is \cite[Lemma 16.9]{DP10} while the second is \cite[(16.3)]{DP10}.

\begin{lem} \label{lem_169dp10}

Let $\C$ and $\C^{\perp}$ be mutually dual $\Fb_b$-linear subspaces of $\Fb_b^{dn}$. Then for any function $f \, : \, \Fb_b^{dn} \longrightarrow \Cx$ we have
\[ \sum_{A \in \C} f(A) = \frac{\# \C}{b^{dn}} \sum_{B \in \C^{\perp}} \hat{f}(B). \]

\end{lem}

\begin{lem} \label{lem_walshduality}

Let $\C$ and $\C^{\perp}$ be mutually dual $\Fb_b$-linear subspaces of $\Fb_b^{dn}$. Let $B \in \Fb_b^{dn}$. Then we have
\[ \sum_{A \in \C} \e^{\frac{2 \pi \im}{b} A \cdot B} = \begin{cases} \# \C, & B \in \C^{\perp}, \\ 0, & B \notin \C^{\perp}. \end{cases} \]

\end{lem}

We will introduce some notation now, slightly changed from what can be found in \cite[16.2]{DP10}. Let $0 \leq \gamma_1, \ldots, \gamma_d \leq n$ be integers. We put $\gamma = (\gamma_1, \ldots, \gamma_d)$. Then we write
\[ \V_{\gamma} = \left\{ A \in \Fb_b^{dn} \, : \, \Phi_n(A) \in \prod_{i = 1}^d \left[ \left. 0, b^{-\gamma_i} \right) \right. \right\}. \]
Hence, $\V_{\gamma}$ consists of all such $A \in \Fb_b^{dn}$ that $a_i = (0, \ldots, 0, a_{i,\gamma_i + 1}, \ldots, a_{i n})$ for all $1 \leq i \leq d$. For all $1 \leq i \leq d$ let $0 \leq \lambda_i \leq \gamma_i$ be integers and let $\lambda = (\lambda_1, \ldots, \lambda_d)$. Then we write $\V_{\gamma, \lambda}$ for the set consisting of all such $A \in \Fb_b^{dn}$ that $a_i = (0, \ldots, 0, a_{i, \lambda_i + 1}, \ldots, a_{i, \gamma_i - 1}, 0, a_{i, \gamma_i + 1}, \ldots, a_{i n})$. The case $\lambda_i = \gamma_i$ is to be understood in the obvious way as $a_i = (0, \ldots, 0, a_{i,\gamma_i + 1}, \ldots, a_{i n})$. Therefore, $\V_{\gamma}^{\perp}$ consists of such $A \in \Fb_b^{dn}$ that $a_i = (a_{i 1}, \ldots, a_{i, \gamma_i}, 0, \ldots, 0)$ and $\V_{\gamma, \lambda}^{\perp}$ consists of such $A \in \Fb_b^{dn}$ that $a_i = (a_{i 1}, \ldots, a_{i, \lambda_i}, 0, \ldots, 0, a_{i, \gamma_i}, 0, \ldots, 0)$.

For a subset $V$ of $\Fb_b^{dn}$ we denote the characteristic function of $V$ by $\chi_V$. The next result is a slight generalization of the corresponding assertion from \cite[Lemma 16.11]{DP10}.

\begin{lem} \label{lem_1611dp10}

Let $\gamma_1, \ldots, \gamma_d, \lambda_1, \ldots, \lambda_d$ be as above. Let $\sigma$ be the number of such $i$ that $\lambda_i < \gamma_i$. For all $B \in \Fb_b^{dn}$ we have
\[ \hat{\chi}_{\V_{\gamma, \lambda}}(B) = b^{dn - |\lambda| - \sigma} \chi_{\V_{\gamma, \lambda}^{\perp}}(B). \]

\end{lem}

The following fact is a generalization of \cite[Lemma 16.13]{DP10}.

\begin{lem} \label{1613dp10}

Let $\C$ and $\C^{\perp}$ be mutually dual $\Fb_b$-linear subspaces of $\Fb_b^{dn}$. Let $\gamma_1, \ldots, \gamma_d, \lambda_1, \ldots, \lambda_d, \sigma$ be as above. Then we have
\[ \# \left( \C \cap \V_{\gamma, \lambda} \right) = \frac{\# \C}{b^{|\lambda| + \sigma}} \, \# \left( \C^{\perp} \cap \V_{\gamma, \lambda}^{\perp} \right). \]

\end{lem}

\begin{prp} \label{main_est_prp}

Let $\C$ be an $\Fb_b$-linear subspace of $\Fb_b^{dn}$ of dimension $n$ with dual space of dimension $dn - n$ satisfying $\delta_n(\C^{\perp}) \geq n + 1$. Let $0 \leq \lambda_i \leq \gamma_i \leq n$ be integers for all $1 \leq i \leq d$ with $|\gamma| \geq n + 1$ and $|\lambda| + d \leq n$. Then we have
\[ \# \left\{ A = (a_1, \ldots, a_d) \in \C^{\perp}: \, v_n(a_i) \leq \gamma_i; \, a_{ik} = 0 \; \forall \, \lambda_i < k < \gamma_i; \, 1 \leq i \leq d \right\} \leq b^d. \]

\end{prp}

\begin{proof}
Let $A \in \C^{\perp}$ with $v_n(a_i) \leq \gamma_i$ and for all $\lambda_i < k < \gamma_i$ with $a_{i k} = 0$ for all $1 \leq i \leq d$. Let $\gamma = (\gamma_1, \ldots, \gamma_d)$ and $\lambda = (\lambda_1, \ldots, \lambda_d)$. Then we have $A \in \V_{\gamma, \lambda}^{\perp}$. Let $\sigma$ be the number of such $i$ that $\lambda_i < \gamma_i$. Analogously to the proof of \cite[Lemma 16.26]{DP10} using Lemma \ref{1613dp10} we get
\begin{align}
\# & \left\{ A = (a_1, \ldots, a_d) \in \C^{\perp}: \, v_n(a_i) \leq \gamma_i; \, a_{ik} = 0 \; \forall \, \lambda_i < k < \gamma_i; \, 1 \leq i \leq d \right\} \notag \\
& \qquad \leq \# \left( \C^{\perp} \cap \V_{\gamma, \lambda}^{\perp} \right) \notag \\
& \qquad = b^{|\lambda| + \sigma - n} \, \# \left( \C \cap \V_{\gamma, \lambda} \right). \label{multline_cardinality}
\end{align}
Now suppose $A \in \V_{\gamma, \lambda}$. Then for all $1 \leq i \leq d$ we have
\[ \Phi_n(a_i) = \frac{a_{i, \lambda_i + 1}}{b^{\lambda_i + 1}} + \ldots + \frac{a_{i, \gamma_i - 1}}{b^{\gamma_i - 1}} + \frac{a_{i, \gamma_i + 1}}{b^{\gamma_i + 1}} + \ldots + \frac{a_{i n}}{b^n} < \frac{1}{b^{\lambda_i}} \]
in the case where $\lambda_i < \gamma_i$ and
\[ \Phi_n(a_i) = \frac{a_{i, \lambda_i + 1}}{b^{\lambda_i + 1}} + \ldots + \frac{a_{i n}}{b^n} < \frac{1}{b^{\lambda_i}} \]
elsewise. Hence, $\Phi_n^d(A)$ is contained in the $b$-adic interval
\[ \prod_{i = 1}^d \left[ \left. 0, b^{-\lambda_i} \right) \right. \]
of volume $b^{-|\lambda|}$. By Proposition \cite[Theorem 7.14]{DP10}, $\Phi_n^d(\C)$ is a digital net in base $b$, and therefore, contains exactly $b^{n - |\lambda|}$ points which lie in a $b$-adic interval of volume $b^{-|\lambda|}$. Therefore, we have
\[ \# \left( \C \cap \V_{\gamma, \lambda} \right) \leq b^{n - |\lambda|} \]
and the result follows from \eqref{multline_cardinality} since $\sigma \leq d$.

\end{proof}

\begin{prp} \label{prp_haar_coeff_cs}

There exists a constant $c > 0$ with the following property. Let $\CS_n$ be a Chen-Skriganov type point set with $N = b^n$ points and let $\mu_{jml}$ be the $b$-adic Haar coefficient of the discrepancy function of $\CS_n$ for $j \in \N_{-1}^d, \, m \in \Dd_j$ and $l \in \B_j$. Then
\begin{enumerate}[(i)]
 	\item if $j = (-1, \ldots, -1)$ then
\[ \left| \mu_{jml} \right| \leq c \, b^{-n}, \] \label{prp_1_cs}
  \item if $j \neq (-1, \ldots, -1)$ and $|j| \leq n$ then
\[ \left| \mu_{jml} \right| \leq c \, b^{-|j| - n}, \] \label{prp_2_cs}
  \item if $j \neq (-1, \ldots, -1)$ and $|j| > n$ and $j_{\eta_1}, \ldots, j_{\eta_s} < n$ then
\[ \left| \mu_{jml} \right| \leq c \, b^{-|j| - n} \] \label{prp_3_cs}
and
\[ \left| \mu_{jml} \right| \leq c \, b^{-2|j|} \]
for all but $b^n$ coefficients $\mu_{jml}$,
  \item if $j \neq (-1, \ldots, -1)$ and $j_{\eta_1} \geq n$ or $\ldots$ or $j_{\eta_s} \geq n$ then
\[ \left| \mu_{jml} \right| \leq c \, b^{-2|j|}, \] \label{prp_4_cs}
\end{enumerate}

\end{prp}

\begin{proof}
Part \eqref{prp_1_cs} is actually Proposition \ref{prp_min1}.

To prove part \eqref{prp_2_cs} we use again the resolution of $D_{\CS_n}$ in 
\[ D_{\CS_n} = \Theta_{\CS_n} + R_{\CS_n}. \]
Let $j \in \N_{-1}^d, \, j \neq (-1, \ldots, -1), \, |j| \leq n, \, m \in \Dd_j, \, l \in \B_j$. The Walsh function series of $h_{jml}$ can be given as
\begin{align} \label{h_j_m_l_given_as}
h_{jml} = \sum_{\alpha \in \N_0^d} \left\langle h_{jml} , \wal_{\alpha} \right\rangle \wal_{\alpha}.
\end{align}
By Lemma \ref{rest_lem} we have
\[ \left| \left\langle R_{\CS_n} , h_{jml} \right\rangle \right| \leq c \, b^{-n} |I_{jm}| = c \, b^{-|j| - n}. \]
We recall that
\[ \left\langle h_{jml}, \wal_{\alpha} \right\rangle = \left\langle h_{j_1 m_1 l_1}, \wal_{\alpha_1} \right\rangle \cdot \ldots \cdot \left\langle h_{j_d m_d l_d}, \wal_{\alpha_d} \right\rangle \]
and
\[ \left\langle \hat{\chi}_{[0,y)}(t) , \wal_{\alpha} \right\rangle = \left\langle \hat{\chi}_{[0,y_1)}(t_1) , \wal_{\alpha_1} \right\rangle \cdot \ldots \cdot \left\langle \hat{\chi}_{[0,y_d)}(t_d) , \wal_{\alpha_d} \right\rangle. \]
We will use Lemmas \ref{lem_scalprod_1} and \ref{lem_scalprod_2} on each of the factors. Lemma \ref{lem_scalprod_1} gives us $|\left\langle h_{j_i m_i l_i}, \wal_{\alpha_i} \right\rangle| \leq b^{-j_i}$ if $j_i \neq -1$ for all $i$. For all $\alpha$ with $\varrho(\alpha_i) \neq j_i + 1$ for some $i$ we have $\left\langle h_{j m l}, \wal_{\alpha} \right\rangle = 0$. We also always get $0$ if the leading digit in the $b$-adic expansion of $\alpha_i$ is not $l_i$ for some $i$. In the case where $j_i = -1$ we can get $b^{-j_i}$, by increasing the constant. From Lemma \ref{lem_scalprod_2} we have $|\left\langle \hat{\chi}_{[0,y_i)}(t_i) , \wal_{\alpha_i} \right\rangle| \leq c \, b^{-\max(\varrho(\alpha_i), \varrho(t_i))}$. Inserting Lemma \ref{theta_lem} and \eqref{h_j_m_l_given_as} we get
\begin{align*}
\left| \mu_{jml} (\Theta_{\CS_n}) \right| & = \left| \left\langle \Theta_{\CS_n} , h_{jml} \right\rangle \right| \\
& = \left| \left\langle \sum_{t \in \Dn'(C_1, \ldots, C_d)} \hat{\chi}_{[0,\cdot)}(t) , \sum_{\alpha \in \N_0^d} \left\langle h_{jml} , \wal_{\alpha} \right\rangle \wal_{\alpha} \right\rangle \right| \\
& \leq \sum_{t \in \Dn'(C_1, \ldots, C_d)} \sum_{\alpha \in \N_0^d} \left| \left\langle \hat{\chi}_{[0,\cdot)}(t) , \wal_{\alpha} \right\rangle \right| \left| \left\langle h_{jml}, \wal_{\alpha} \right\rangle \right| \\
& \leq c_1 \, b^{-j_1 - \ldots - j_d} \sum_{t \in \Dn'(C_1, \ldots, C_d)} b^{-\max(j_1, \varrho(t_1)) - \ldots - \max(j_d, \varrho(t_d))}.
\end{align*}
The summation in $\alpha$ disappears due to the following facts. The application of Lemma \ref{lem_scalprod_1} leaves only all such $\alpha$ with $\varrho(\alpha_i) = j_i + 1$ and with $l_i$ as leading digit in the $b$-adic expansion of $\alpha_i$ for all $i$. The application of Lemma \ref{lem_scalprod_2} leaves then at most one $\alpha$ per $t$, namely the one with either $\alpha_i = t_i'$ (if $\varrho(t_i) > j_i + 1$) or $\alpha_i = t_i + l_i \, b^{j_i}$ (if $\varrho(t_i) \leq j_i + 1$) for all $i$. In the cases where there is an $i$ with $\varrho(t_i) > j_i + 1$, it is possible that no $\alpha$ is left in the summation, since we still have the condition on $\alpha_i$ that the leading digit in the $b$-adic expansion is $l_i$, which cannot be guaranteed for $t_i'$.

Our next step is to break the sum above into sums where for every $t$ every coordinate either has bigger NRT weight than the corresponding coordinate of $j$ or a smaller NRT weight. Let $0 \leq r \leq d$ be the integer that is the cardinality of such $1 \leq i \leq d$ that the NRT weight is smaller. Without loss of generality we consider for every $r$ only the case where for $1 \leq i \leq r$ we have $\varrho(t_i) \leq j_i$ while for $r + 1 \leq i \leq d$ we have $\varrho(t_i) > j_i$. All the other cases follow from renaming the indices and we will just increase the constant. In the notation we split the sum
\[ \sum_{t \in \Dn'(C_1, \ldots, C_d)} \leq c_2 \, \sum_{r = 0}^d \, \sum_{t \in \Dn'_r(C_1, \ldots, C_d)} \]
where by $\Dn'_r(C_1, \ldots, C_d)$ we mean the subset of $\Dn'(C_1, \ldots, C_d)$ according to what we explained above (with ordered indices and other cases incorporated into the constant, $r$ coordinates have smaller NRT weight). So we have
\begin{align*}
\left| \mu_{jml} (\Theta_{\CS_n}) \right| & \leq c_3 \, b^{-j_1 - \ldots - j_d} \sum_{r = 0}^d \, \sum_{t \in \Dn'_r(C_1, \ldots, C_d)} b^{-j_1 - \ldots - j_r - \varrho(t_{r + 1}) - \ldots - \varrho(t_d)} \\
& = c_3 \, \sum_{r = 0}^d \, b^{-2j_1 - \ldots - 2j_r - j_{r + 1} - \ldots - j_d} \sum_{t \in \Dn'_r(C_1, \ldots, C_d)} b^{-\varrho(t_{r + 1}) - \ldots - \varrho(t_d)}.
\end{align*}
Instead of summing over $t$, we can sum over the values of $\varrho(t)$, considering the number of such $t$ that $\varrho(t_i) = \gamma_i$, $1 \leq i \leq d$.
We recall that $\CS_n = \Phi_n^d(\C_n)$. Then we denote
\[ \omega_{\gamma} = \# \left\{ A \in \C_n^{\perp}: \, v_n(a_i) = \gamma_i \, \forall \, i \, \wedge \, a_{ik} = 0 \; \forall \, j_i < k < \gamma_i; \, r + 1 \leq i \leq d \right\} \]
and 
\[ \tilde{\omega}_{\gamma} = \# \left\{ A \in \C_n^{\perp}: \, v_n(a_i) \leq \gamma_i \, \forall \, i \, \wedge \, a_{ik} = 0 \; \forall \, j_i < k < \gamma_i; \, r + 1 \leq i \leq d \right\}. \]
Let $\Gamma$ consist of all such $\gamma = (\gamma_1, \ldots, \gamma_d)$ that $0 \leq \gamma_i \leq j_i$ for $1 \leq i \leq r$, $j_i < \gamma_i \leq n$ for $r + 1 \leq i \leq d$ and $|\gamma| \geq n + 1$. Then we have
\[ \left| \mu_{jml} (\Theta_{\CS_n}) \right| \leq c_3 \, \sum_{r = 0}^d \, b^{-2j_1 - \ldots - 2j_r - j_{r+1} - \ldots - j_d} \sum_{\gamma \in \Gamma} b^{-\gamma_{r + 1} - \ldots - \gamma_d} \, \omega_{\gamma}. \]
We can apply Proposition \ref{main_est_prp} with $\lambda_i = \gamma_i, \, 1 \leq i \leq r$ and $\lambda_i = j_i, \, r + 1 \leq i \leq d$. Thereby, we get $\tilde{\omega}_{\gamma} \leq b^d$. An obvious observation is that
\[ \sum_{0 \leq \kappa_i \leq \gamma_i, \, 1 \leq i \leq d} \omega_{\kappa} \leq \tilde{\omega}_{\gamma} \]
with $\kappa = (\kappa_1, \ldots, \kappa_d)$. Recall the notation $\bar{n} = (n, \ldots, n)$. For all $\gamma \in \Gamma$ it holds that $-\gamma_{r + 1} - \ldots - \gamma_d \leq \gamma_1 + \ldots + \gamma_r - n - 1$ so,
\begin{align*}
& \left| \mu_{jml} (\Theta_{\CS_n}) \right| \leq c_3 \, \sum_{r = 0}^d \, b^{-2j_1 - \ldots - 2j_r - j_{r+1} - \ldots - j_d} \sum_{\gamma \in \Gamma} b^{-n - 1 + \gamma_1 + \ldots + \gamma_r} \, \omega_{\gamma} \\
& \leq c_4 \, \sum_{r = 0}^d \, b^{-2j_1 - \ldots - 2j_r - j_{r+1} - \ldots - j_d - n} \sum_{0 \leq \gamma_i \leq j_i, \, 1 \leq i \leq r} b^{\gamma_1 + \ldots + \gamma_r} \sum_{j_i < \gamma_i \leq n, \, r + 1 \leq i \leq d} \omega_{\gamma} \\
& \leq c_4 \, \sum_{r = 0}^d \, b^{-2j_1 - \ldots - 2j_r - j_{r+1} - \ldots - j_d - n} \, \prod_{i = 1}^r \sum_{\kappa_i = 0}^{j_i} b^{\kappa_i} \sum_{j_i < \gamma_i \leq n, \, r + 1 \leq i \leq d} \, \max_{0 \leq \gamma_i \leq j_i, \, 1 \leq i \leq r} \omega_{\gamma} \\
& \leq c_5 \, \sum_{r = 0}^d \, b^{-j_1 - \ldots - j_d - n} \sum_{0 \leq \gamma_i \leq n, \, 1 \leq i \leq d} \omega_{\gamma} \\
& \leq c_6 \, b^{-j_1 - \ldots - j_d - n} \, \tilde{\omega}_{\bar{n}} \\
& \leq c_6 \, b^{-j_1 - \ldots - j_d - n} \, b^d \\
& \leq c_7 \, b^{-|j| - n}.
\end{align*}

For the part \eqref{prp_3_cs} let $|j| > n$ and $j_{\eta_1}, \ldots, j_{\eta_s} < n$. We recall that $\CS_n$ contains exactly $N = b^n$ points and for fixed $j \in \N_{-1}^d$, the interiors of the $b$-adic intervals $I_{jm}$ are mutually disjoint. There are no more than $b^n$ such $b$-adic intervals which contain a point of $\CS_n$ meaning that all but $b^n$ intervals contain no points at all. This fact combined with Lemma \ref{lem_haar_coeff_besov_x} gives us the second statement of this part. The remaining boxes contain exactly one point of $\CS_n$. So from Lemmas \ref{lem_haar_coeff_besov_x} and \ref{lem_haar_coeff_besov_indicator} we get the first statement of this part.

Finally, let $j_{\eta_1} \geq n$ or $\ldots$ or $j_{\eta_s} \geq n$, then there is no point of $\CS_n$ which is contained in the interior of the $b$-adic interval $I_{jm}$. Thereby part \eqref{prp_4_cs} follows from Lemma \ref{lem_haar_coeff_besov_x}.

\end{proof}

\section{The proof of the main result}
\begin{proof}[Proof of Theorem \ref{thm_main}]
The point set satisfying the assertion is the Chen-Skriganov type point set $\CS_n$. Let $\mu_{jml}$ be the $b$-adic Haar coefficients of the discrepancy function of $\CS_n$. We write $|j| = j_{\eta_1} + \ldots + j_{\eta_s}$. We have an equivalent quasi-norm on $S_{pq}^r B([0,1)^d)$ in \eqref{eq_quasinorm} so that the proof of the inequality
\[ \left( \sum_{j \in \N_{-1}^d} b^{|j|(r - \frac{1}{p} + 1) q} \left( \sum_{m \in \Dd_j, \, l \in \B_j} |\mu_{jml}|^p \right)^{\frac{q}{p}} \right)^{\frac{1}{q}} \leq C \, b^{n(r-1)} n^{\frac{d-1}{q}} \]
for some constant $C > 0$ establishes the proof of the theorem in this case. 

To estimate the expression on the left-hand side, we use Minkowski's inequality to split the sum into summands according to the cases of Proposition \ref{prp_haar_coeff_cs}. We denote
\[ \Xi_j = b^{|j|(r - \frac{1}{p} + 1) q} \left( \sum_{m \in \Dd_j, \, l \in \B_j} |\mu_{jml}|^p \right)^{\frac{1}{p}} \]
and get
\[ \left( \sum_{j \in \N_{-1}^d} \Xi_j^q \right)^{\frac{1}{q}} \leq \Xi_{(-1, \ldots, -1)} + \sum_{s = 1}^d \left[ \left( \sum_{j \in J_s^1} \Xi_j^q \right)^{\frac{1}{q}} + \left( \sum_{j \in J_s^2} \Xi_j^q \right)^{\frac{1}{q}} + \sum_{i = 1}^s \left( \sum_{j \in J_{s i}^3} \Xi_j^q \right)^{\frac{1}{q}} \right] \]
where $J_s^1$ is the set of all such $j \neq (-1, \ldots, -1)$ for which $|j| \leq n$, $J_s^2$ is the set of all such $j \neq (-1, \ldots, -1)$ for which $0 \leq j_{\eta_1}, \ldots, j_{\eta_s} \leq n-1$ and $|j| > n$ and $J_{s i}^3$ is the set of all such $j$ for which $j_{\eta_i} \geq n$.

We will show that each of the summands above can be bounded by $C \, b^{n(r-1)} n^{\frac{d-1}{q}}$ which finishes the proof.

Part \eqref{prp_1_cs} of Proposition \ref{prp_haar_coeff_cs} gives us for $j = (-1, \ldots, -1), \, m = (0,\ldots,0), \, l = (0,\ldots,0)$
\[ \Xi_j = |\mu_{jml}| \leq c_1 b^{-n} \leq c_2 b^{n(r-1)} n^{\frac{d-1}{q}}. \]
Let now $1 \leq s \leq d$. We will use \eqref{prp_2_cs} in Proposition \ref{prp_haar_coeff_cs} and Lemma \ref{lem_lambda_s_minus_1}. The summation over $l \in \B_j$ can be incorporated into the constant and we recall that $\# \Dd_j = b^{|j|}$. Hence (using the fact that $r < 0$) we have
\begin{align*}
\left( \sum_{j \in J_s^1} \Xi_j^q \right)^{\frac{1}{q}} & \leq c_3 \left( \sum_{j \in J_s^1} b^{|j|(r - \frac{1}{p} + 1) q} \left( \sum_{m \in \Dd_j} b^{(-|j| - n) p} \right)^{\frac{q}{p}} \right)^{\frac{1}{q}} \\
& = c_3 \left( \sum_{j \in J_s^1} b^{(|j| r - n ) q} \right)^{\frac{1}{q}} \\
& \leq c_4 \left( \sum_{\lambda = 0}^n b^{(\lambda r - n) q} (\lambda + 1)^{s - 1} \right)^{\frac{1}{q}} \\
& \leq c_5 \, n^{\frac{s - 1}{q}} \, b^{-n} \left( \sum_{\lambda = 0}^n b^{\lambda r q} \right)^{\frac{1}{q}} \\
& \leq c_6 \, n^{\frac{d - 1}{q}} \, b^{n (r-1)}.
\end{align*}
From \eqref{prp_3_cs} in the same proposition (using the fact that $r - \frac{1}{p} < 0$ and $r - 1 \leq 0$) we have
\begin{align*}
\left( \sum_{j \in J_s^2} \Xi_j^q \right)^{\frac{1}{q}} & \leq c_7 \left( \sum_{j \in J_s^2} b^{|j|(r - \frac{1}{p} + 1) q} \, b^{n \frac{q}{p}} \, b^{(-|j| - n) q} \right)^{\frac{1}{q}} \\
& \quad + c_8 \left( \sum_{j \in J_s^2} b^{|j|(r - \frac{1}{p} + 1)q} \, b^{|j| \frac{q}{p}} \, b^{-2|j| q} \right)^{\frac{1}{q}} \\
& = c_7 \left( \sum_{j \in J_s^2} b^{\left[ |j| (r - \frac{1}{p}) + \frac{n}{p} - n \right] q} \right)^{\frac{1}{q}} \\
& \quad + c_8 \left( \sum_{j \in J_s^2} b^{|j|(r - 1)q} \right)^{\frac{1}{q}}\\
& \leq c_7 \left( \sum_{\lambda = n+1}^{s(n - 1)} (\lambda + 1)^{s-1} b^{\left[ \lambda(r - \frac{1}{p}) + \frac{n}{p} - n \right] q} \right)^{\frac{1}{q}} \\
& \quad + c_8 \left( \sum_{\lambda = n+1}^{s(n - 1)} (\lambda + 1)^{s-1} b^{\lambda (r-1) q} \right)^{\frac{1}{q}} \\
& \leq c_9 \, n^{\frac{s - 1}{q}} b^{\frac{n}{p} - n} \left( \sum_{\lambda = n + 1}^{s(n - 1)} b^{\lambda (r - \frac{1}{p}) q} \right)^{\frac{1}{q}} + c_{10} \, n^{\frac{s - 1}{q}} \left( \sum_{\lambda = n + 1}^{s(n - 1)} b^{\lambda (r - 1) q} \right)^{\frac{1}{q}} \\
& \leq c_{11} \, n^{\frac{s - 1}{q}} b^{\frac{n}{p} - n} \, b^{n (r - \frac{1}{p})} + c_{12} \, n^{\frac{s - 1}{q}} b^{n (r - 1)} \\
& \leq c_{13} \, n^{\frac{d - 1}{q}} \, b^{n(r-1)}.
\end{align*}
Part \eqref{prp_4_cs} in Proposition \ref{prp_haar_coeff_cs} gives us for any $1 \leq i \leq s$
\begin{align*}
\left( \sum_{j \in J_{s i}^3} \Xi_j^q \right)^{\frac{1}{q}} & \leq c_{14} \left( \sum_{j \in J_{s i}^3} b^{|j|(r - \frac{1}{p} + 1) q} \, b^{|j| \frac{q}{p}} \, b^{-2|j| q} \right)^{\frac{1}{q}}\\
& \leq c_{15} \left( \sum_{\lambda = n}^\infty (\lambda + 1)^{s-1} b^{\lambda(r - 1) q} \right)^{\frac{1}{q}}\\
& \leq c_{16} n^{\frac{d-1}{q}} b^{n(r-1)}.
\end{align*}

The cases $p = \infty$ and $q = \infty$ have to be modified in the usual way.

\end{proof}

\end{document}